\documentclass{amsart}
\usepackage[english]{babel}
\usepackage{amsthm}
\usepackage{graphicx} 
\usepackage{amsmath}
\usepackage{amssymb}
\usepackage{xcolor}
\usepackage{verbatim} 
\usepackage{bbm}
\usepackage{dsfont}
\usepackage{hyperref}
\usepackage{bbold}

\def\R{\mathbb{R}}

\def\S{\mathbb{S}}

\def\I{\mathcal{I}}

\def\J{\mathcal{J}}
\def\Jj{\frak{J}}

\def\1{\mathbb{1}}
\def\grad{\nabla}

\def\d{~\mathrm{d}}

\def\ln{\mathrm{ln}}

\def\sin{\mathrm{sin}}
\def\cos{\mathrm{cos}}

\def \det{\mathrm{det}}

\def\la{\langle}
\def\lla{\left\langle}
\def\ra{\rangle}
\def\rra{\right\rangle}

\def\dt{\frac{\mathrm{d}}{\mathrm{d}t}}

\def\p{\partial}

\def\1{\mathbb{1}}
\def\nyx{\frac{y-x}{|y-x|}}
\def\jxy{J(x-y,v-w)}

\newtheorem{theorem}{Theorem}[section]
\newtheorem{lemma}{Lemma}
\newtheorem{corollary}{Corollary}

\newtheorem*{definition}{Definition}

\title[Conservative Formulations: Enskog \& Povzner Equations]{Conservative Formulations of the Standard Enskog and Povzner Equations}
\author{Zhe Chen}
\address{Laboratoire MAP5, Université Paris Cité, 45 rue des Saints-Pères, 75270 Paris Cedex 06, France}
\email{zhe.chen@u-paris.fr}

\begin{document}

\begin{abstract}
This article introduces a conservative formulation of the Standard Enskog equation and the Povzner equation, both of which generalize the Boltzmann equation by incorporating the contribution of particle volume in collisions. The primary result expresses these collision integrals as the divergence with respect to the velocity variable $v$  of a mass current. Moreover, the terms $v C[f,f]$ and $|v|^2 C[f,f]$, where $C[f,f]$ denotes the Standard Enskog or Povzner collision integral, are represented as phase-space divergences—i.e., divergences in both position and velocity—of corresponding momentum and energy currents. This work extends Villani's earlier result [Math. Modelling Numer. Anal. M2AN 33 (1999), 209–227] for the classical Boltzmann equation to the case of dense gases.
\end{abstract}

\maketitle

\section*{Introduction}
This paper focuses on collisional kinetic models designed to describe the dynamics of monatomic gases. The most classical example is the Boltzmann equation, introduced by Ludwig Boltzmann in 1872, which models the statistical behavior of dilute gas particles through binary collisions. The Standard Enskog equation, proposed by David Enskog in 1921, extends Boltzmann's framework to dense gases by accounting for particle size and excluded volume. The Povzner collision integral, introduced in 1965, generalizes the model by allowing collisions between molecules to be elastic or soft. Roughly speaking, Enskog's model permits collisions on a sphere, while Povzner's model extends this to collisions within a spherical shell.

The kinetic equation is connected to the conservation laws of mass, momentum, and energy through the definition of macroscopic quantities via the particle distribution function; see Section 3.3 in \cite{cercignani1994}, for example. At the time of this writing, the direct verification of local conservation laws for mass, momentum, and energy by classical solutions of the Standard Enskog equation or the Povzner equation, as discussed in Section \ref{sec:collisional kinetic models}, does not seem to have been studied systematically. To address this gap, we propose a systematic method for deriving local conservation laws in the framework of kinetic models for these two equations. Our approach is based on expressing the collision integrals as divergences in the position and velocity variables of mass, momentum, and energy currents, drawing inspiration from a lesser-known idea by Landau regarding the Boltzmann collision integral; see, for instance, \cite{villani1999,Landau10}.

The outline of this paper is as follows: In Section \ref{sec:collisional kinetic models}, we first revisit the classical Boltzmann equation and introduce the motivation behind this work. We then present the two models under consideration—the Standard Enskog equation and the Povzner equation. Section \ref{sec:Main result} is dedicated to the main results of the paper. Subsequently, in Section \ref{sec:Application}, we explore the application of these results to macroscopic dynamics and local conservation laws. Finally, the proofs of the main results are provided in Section \ref{sec:proof of Enskog} and Section \ref{sec:proof of Povzner}, respectively. Similar results for a class of delocalized collision integrals which does not include the Standard Enskog and the Povzner collision integrals have been obtained in \cite{CCG2024}.
\section{Collisional kinetic Models}\label{sec:collisional kinetic models}
\subsection{The Boltzmann equation}
We begin with considering the following collisional kinetic equation 
\begin{equation}\label{genral kinetic equation}
    (\p_t+v\cdot\grad_x)f(t,x,v)=C[f,f](t,x,v),
\end{equation}
where $f(t,x,v)\in \R_+$ is a single particle distribution depending on time $t\in \R_+$ space $x=(x_1,x_2,x_3)\in\R^3$, and velocity $v=(v_1,v_2,v_3)\in\R^3$. The most classical example is the hard sphere Boltzmann collision operator defined by
\begin{equation}\label{boltzmann collision operator}
    B[f,g]:=\int_{\R^d\times\S^{d-1}}(f(x,v')g(x,w')-f(x,v)g(x,w))\la v-w,n\ra_+\d n\d w,
\end{equation}
where $\la\cdot,\cdot\ra$ denotes the inner product in Euclidean space, and $\la\cdot,\cdot\ra_+:=\max\{0,\la\cdot,\cdot\ra\}$. The pre-collision velocities $(v',w')$ are determined by the unit vector $n\in \S^{2}$ 
\begin{equation}\label{pre-to-post collision transformation}
    \begin{cases}
    v'(v,w,n)=v-n\la n,v-w\ra\\
    w'(v,w,n)=w+n\la n,v-w\ra,
    \end{cases}
\end{equation}

Here we list some basic properties of the collision transform;
\begin{itemize}
    \item The collision transform is reversible:
    \begin{equation}\label{reversible collision}
        v''(v'(v,w,n),w'(v,w,n),n)=v;\quad w''(v'(v,w,n),w'(v,w,n),n)=w.
    \end{equation}
    \item The collision transform is symmetric:
    \begin{equation}\label{symmetry collision}
        v'(w,v,n)=w'(v,w,n);\quad w'(w,v,n)=v'(v,w,n).
    \end{equation}
    \item The collision transform is even with respect to $n$:
    \begin{equation}\label{collision is even wrsp n}
        v'(w,v,-n)=v'(v,w,n);\quad w'(w,v,-n)=w'(v,w,n).
    \end{equation}
    \item collision transform has Jacobian determinant:
    \begin{equation}\label{Jacobian of collision}
       \det\left(\frac{\p (v',w')}{\p(v,w)}\right)=-1.
    \end{equation}
    \item The collision transform conserves momentum and energy:
    \begin{equation}\label{collision conserves momentum and energy}
        v'(w,v,n)+w'(v,w,n)=v+w;\quad |v'(x,w,n)|^2+|w'(v,w,n)|^2=|v|^2+|w|^2.
    \end{equation}
     \item The collision transform reverses the direction with respect to $n$:
    \begin{equation}\label{post-collision inverses the direction}
        \la v'(v,w,n)-w'(v,w,n),n\ra=-\la v-w,n\ra
    \end{equation}
\end{itemize}
The two particles collide at the same position $x$. Due to the pre-to-post transformation \eqref{pre-to-post collision transformation}, the Boltzmann collision operator locally conserves the mass, momentum, and energy respectively, namely for classical solution with rapidly decaying of \eqref{genral kinetic equation} with Boltzmann collision integrals,
\begin{equation}\label{convervation of boltzmann collision operator}
 \dt\int_{\R^3} f(t,x,v)
\begin{pmatrix}
1\\
v_1\\
v_2\\
v_3\\
|v|^2
\end{pmatrix}\d v=
    \int_{\R^3} {B}[f,f]
\begin{pmatrix}
1\\
v_1\\
v_2\\
v_3\\
|v|^2
\end{pmatrix}\d v=\mathbf{0}.
\end{equation}
This result is given by the weak formula for the Boltzmann collision interal which can be found in every textbook in kinetic theory, such as section 3.1 in \cite{cercignani1994}. 
If we generalize the model to dense gases, which means the molecular radius cannot be neglected.
In that case, the centers of two colliding particles are no longer at the same position in space.

 For Enskog's model and Povzner's model, it is easily seen that one cannot arrive the same local result as in \eqref{convervation of boltzmann collision operator}. In other words, the quantities shown in \eqref{convervation of boltzmann collision operator} do not always vanish. On the other hand, we have the global conservation 
\begin{equation}\label{global conservation of collision operator}
   \int_{\R^3\times \R^3} E[f,f]
\begin{pmatrix}
1\\
v_1\\
v_2\\
v_3\\
|v|^2
\end{pmatrix}\d v\d x= \int_{\R^3\times \R^3} P[f,f]
\begin{pmatrix}
1\\
v_1\\
v_2\\
v_3\\
|v|^2
\end{pmatrix}\d v\d x=0.
\end{equation}
This result for the Povzner collision operator can be found in \cite{FORNASIER2011}. For the Enskog version, the reader can read Chapter 16 in \cite{chapman1990}. 
In \cite{Landau10} (see \cite{villani1999} for a rigorous discussion), it was found that the Boltzmann collision integral can be expressed as the divergence in $v$ of a mass current, $B[f,f]=-\grad_v\cdot \mathcal{J}_0(f,f)$ where 
\[
\mathcal{J}_0(f,f):=-\int_{\la v-w,n\ra>0}\int^{\la v-w,n\ra}_0\la v-w,n\ra f(x,v+sn)f(x,w+sn)n\d n\d w\d s.
\]
In [PL12], this formula is used to derive the Landau equation from the Boltzmann equation. Comparing to the delocalized collision operator, a natural question comes from \eqref{global conservation of collision operator} : can we express delocalized collision operator as the following conservative form
\begin{equation}\label{general conservative form}
C(f,f)\left(
\begin{aligned}
&1\\
&v_1\\
&v_2\\
&v_3\\
&|v|^2
\end{aligned}\right)
= \nabla_v \cdot \mathbf{J} + \nabla_x \cdot \mathbf{I}  ?
\end{equation}

\subsection{The Standard Enskog equation}
We write the following so-called Standard Enskog equation
\begin{equation}
    (\p_t+v\cdot\grad_x)f(t,x,v)=E[f,f](t,x,v),
\end{equation}
where Enskog collision integral is defiend as
\begin{multline}\label{Standard Enskog collision integral}
    E[f,g](t,x,v):=\int_{\R^3\times\S^2}\bigg[\chi(x-\frac{1}{2}\sigma n)f(t,x,v')g(t,x-\sigma n,w')\\
    -\chi(x+\frac{1}{2}\sigma n)f(t,x,v)g(t,x+\sigma n,w)\bigg]\sigma^2\la v-w,n\ra_+\d n \d w,
\end{multline}
the pre-collision velocities $(v',w')$ are given by \eqref{pre-to-post collision transformation}.

We have to say that it seems nonsensical to define Standard Enskog collision integral as a bilinear functional of distribution functions, since one always assumes that the correlation function $\chi$ depends on one of the distribution function $f$. In that case, the two functions don't share the symmetry structure anymore. Here, we write the collision integral in this way only for mathematical convenience.  

In addition to Standard Enskog equation, some people have also proposed a Revised Enskog equation. It has more mathematical structure. For example, it has the symmetry between two colliding particles, and it is more easily generalized to the case of mixtures. (The standard Enskog equation can also be generated to mixture case, see Chapter 16 of \cite{chapman1990} for instance.)
The Revised Enskog collision operator is
          \begin{align*}\label{Revised Enskog collision integral}
           \Tilde{E}[f,g]:=&\int_{\R^3\times\S^2}\bigg( Y\big(\rho_f(x),\rho_g(x-\sigma n)\big)f(x,v')g(x-\sigma n,w')\nonumber\\
           &-Y\big(\rho_f(x),\rho_g(x+\sigma n)\big)f(x,v)g(x+\sigma n,w)\bigg)\la n,v-w\ra_{+}\d w\d n,
           \end{align*}
           where the correlation function $Y$ satisfies $Y(x,y)=Y(y,x)$.
           The Enskog equation is a generalized Boltzmann equation aimed at modelling dense gases. When we formally derive the Boltzmann equation from Liouville equation,
            we no longer consider the particles as mass points. We distinguish between the collision point and the particle center. Its advantage is to describe two different kinds of particles collisions.
           The Enskog equation was first proposed by David Enskog in 1922 \cite{Enskog1922}. \cite{VANBEIJEREN1973225}\cite{VANBEIJEREN1973437} proposed the modified Enskog equation. 
           More research has been conducted on the revised Enskog equation due to its improved mathematical structure, particularly its symmetry, which makes it easier to establish properties such as the H-theorem \cite{Resibois1978}, and consistents with irreversible thermodynamics-- leads to the correct equilibrium \cite{VanBeijeren1983}. Nevertheless, the standard Enskog equation remains equally valuable for study because of its more intuitive physical interpretation. Specifically, the correlation function in the standard Enskog equation depends on the ratio:
\[
\frac{\text{Impact of Volume Occupation}}{\text{Shielding Influence}}.
\]
Hence, we have a more explicit formula for analyzing the equation. For example, asymptotically, one can assume $\chi(t,x)=1+\frac{5}{8}\frac{2}{3}\pi\sigma^3\rho_f(t,x)$ where $\rho_f(t,x):=\int_{\R^3}f(t,x,v)\d v$: see for instance Chapter 16 in \cite{chapman1990}.

In the following discussion, we focus on the Standard Enskog model. For details regarding the revised/modified Enskog model, is referred to \cite{CCG2024}. It is important to emphasize that the results presented in \cite{CCG2024} do not include those in Section \ref{sec:Main result}, despite the similarities in the approaches. This is because, in the proofs of \cite{CCG2024}, the collision integral $\mathrm{St}[\cdot,\cdot]$ is assumed to be in a symmetric form, i.e., $\mathrm{St}[f,g] = \mathrm{St}[g,f]$.

\subsection{The Povzner equation}
The Povzner collision integral was proposed to study the delocalized collision by Povzner in \cite{povzner1965}. The collision integral is of the following form
          \begin{multline}\label{Povzner collision integral}
            P[f,g]:=\int_{\R^3\times \R^3}\bigg( f(x,v')g(y,w')-f(x,v)g(y,w)\bigg)J(x-y,v-w)\d w\d y ,
          \end{multline}
          where the pre-collision velocities $(v',w')$ are also given by \eqref{pre-to-post collision transformation} with $n=\nyx$. $J(\xi,v)$ is assumed to satisfy the following four assumptions as showed in \cite{povzner1965}
          \begin{enumerate}
              \item[(i)] $J$ is continuous, and $|J(\xi,v)|\lesssim |v|+1$.
              \item[(ii)] Time reversibility: 
              \begin{equation}\label{Povzner collision kernel is time reversible}
                  J(-\xi,-v)=J(\xi,v).
              \end{equation}
              \item[(iii)] Invariance after collision: 
              \begin{equation}\label{Povzner kernel is invariant after collision}J(\xi,v'-w')=J(\xi,v-w).\end{equation}
              \item[(iv)] Short-range interaction: There exists a constant $R$ such that $J(\xi,v)=0$ for $|\xi|>R$.
          \end{enumerate}
          An example of Povzner kernel in \cite{FORNASIER2011} is 
          \[
          J(\xi,v):=\frac{1}{2\delta^3}\frac{1}{\theta}H(\delta-|\xi|)H(\theta-|v|)\left|v\cdot\frac{\xi}{|\xi|}\right|,
          \]
          where $H$ is the Heaviside step function.
          From this example, we can find out more clearly that the Povzner collision integral allows for the possibility of a spatial smearing process with range of $\theta$ and $\delta$ of collisions as well as for the classically assumed process in which the molecules collide at a point.
          We point out here that if we set $J(\xi,v)=\Tilde{J}(|v|,\frac{\la \xi,v\ra}{|\xi||v|})\delta(|\xi|)$ in \eqref{Povzner collision integral}, then we can recover the Boltzmann collision kernel. 
         However, the Povzner collision integral generalizes the following form of the Boltzmann collision integral:

\[
\frac{1}{2}\int_{\R^3 \times \S^2} [f(x,v')f(x,w') - f(x,v)f(x,w)] |\langle v-w, n \rangle|  \d n \d w,
\]

which corresponds to \eqref{boltzmann collision operator} after changing $n$ into $-n$ and symmetrizing \eqref{boltzmann collision operator}. In the Boltzmann collision operator, the direction of the unit vector $n$ is inconsequential, as it cancels out when taking the Boltzmann-Grad limit. However, in the case of dense gases, the locations of the gain term and the loss term differ. Consequently, a more physically meaningful assumption would be:

\begin{equation}\label{new assumption for Povzner collision kernel}
J(\xi, v'-w') = J(-\xi, v-w),
\end{equation}

instead of \eqref{Povzner kernel is invariant after collision}. Nevertheless, in the following discussion, we adhere to Povzner's original assumption \eqref{Povzner kernel is invariant after collision}. For results derived under the new assumption \eqref{new assumption for Povzner collision kernel}, we refer the reader to \cite{CCG2024}.
\begin{figure}[h]
    \centering
    \includegraphics[width=0.7\textwidth]{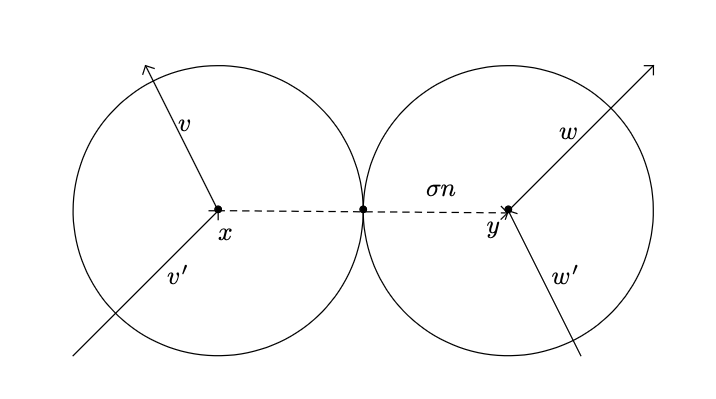}
    \caption{Schematic diagram of ingoing and outgoing processes}
    \label{fig:ingoing outgoing}
\end{figure}

\section{Main results}\label{sec:Main result}
As mentioned in section \ref{sec:collisional kinetic models}, we seek a representation of a form of \eqref{general conservative form}. 
Before stating the main result, we introduce the general Landau mass current
\begin{definition}[Landau mass current for the Standard Enskog collision integral] For each $f,g\in C(\R^3\times\R^3)$ with rapidly decaying, we define
\begin{equation}\label{Landau current for the Standard Enskog collision integral}
    \Jj_0^E[f,g](x,n,v):=\int_{\R^3}\int^{\la v-w,n\ra}_0\chi(x+\frac{1}{2}\sigma n)f(x,v+s n)g(x+\sigma n,w+s n)n\d s\d w
\end{equation}

\end{definition}
\begin{definition}[Landau mass current for the Povzner collision integral] For each $f,g\in C(\R^3\times\R^3)$ with rapidly decay, we define
\begin{multline}\label{Landau current for the Povzner collision integral}
     \Jj_0^P[f,g](x,y,v):=\int_{\R^3}\int^{\la v-w,n\ra}_0f(x,v+s \frac{y-x}{|y-x|})g\left(y,w+s \frac{y-x}{|y-x|}\right)\\
     \times J(x-y,v-w)\frac{y-x}{|y-x|}\d s\d w
\end{multline}

\end{definition}
For Enskog collision integral, we have the following result
\begin{theorem}\label{thm:Enskog}
For each rapidly decaying function $f\in C(\R^3\times \R^3)$, we have 
\begin{align}
     &E[f,f]=\grad_v\cdot \J_0^E[f,f]\label{Mass current for Enskog collision integral}\\
     &E[f,f]v_1=\grad_v\cdot \J_1^E[f,f]+\grad_x\cdot \I_1^E[f,f]\label{Momentum current 1 for Enskog collision integral}\\
     &E[f,f]v_2=\grad_v\cdot \J_2^E[f,f]+\grad_x\cdot \I_2^E[f,f]\label{Momentum current 2 for Enskog collision integral}\\
     &E[f,f]v_3=\grad_v\cdot \J_3^E[f,f]+\grad_x\cdot \I_3^E[f,f]\label{Momentum current 3 for Enskog collision integral}\\
     &E[f,f]|v|^2=\grad_v\cdot \J_4^E[f,f]+\grad_x\cdot \I_4^E[f,f]\label{Energy current for Enskog collision integral}
\end{align}
where vectors fields $\I_k^E(x,v),\J_l^E(x,v)$, $k=1,2,3,4$, $l=0,1,2,3,4$ are defined by
\begin{equation}\label{definition of mass current for Enskog collision integral}
    \J_0^E=\int_{\S^2}\Jj_0^E[f,f](x,n,v)\d n,
\end{equation}
for $k=1,2,3$:
\begin{align}\label{definition of momentum current in v for Enskog collision integral}
    &\J_k^E[f,f]\\\nonumber
    =&\J_0^E[v_kf,f]-\int_{\S^2}\Jj_0^E[\la v,n\ra n_kf,f]\d n+\int_{\S^2}\Jj_0^E[f,\la w,n\ra n_kf]\d n\\\nonumber
    &-\frac{1}{2}\int_{(\R^3)^4}\int^{\pi/2}_0\chi(x-\frac{1}{2}\sigma n)f(x-\sigma n,v_{-\theta})f(x,w_{-\theta})\la v_{-\theta}-w_{-\theta},n\ra_+^2n_k\\\nonumber
    &\qquad\times w\d n\d w\d \theta,
\end{align}
and $l=1,2,3$
\begin{multline} \label{definition of momentum current in x for Enskog collision integral}
    \I_l^E[f,f]=\frac{1}{2} \int_{\R^3\times \S^2}\int^\sigma_0\chi\left(x+(\frac{1}{2}\sigma-s) n\right)f(x-sn,v)f(x+(\sigma-s) n,w)\\
    \qquad\times\la v-w,n\ra_+^2n_ln\d w\d n\d s,
\end{multline}
finally,
\begin{align}\label{definition of energy current in v for Enskog collision integral}
    &\J_4^E[f,f]\\\nonumber
    =&\J_0^E[|v|^2f,f]-\int_{\S^2}\Jj_0^E[\la v,n\ra^2f,f]\d n+\int_{\S^2}\Jj_0^E[f,\la w,n\ra^2f]\d n\\\nonumber
    &-\frac{1}{2}\int_{(\R^3)^4}\int^{\pi/2}_0\chi(x-\frac{1}{2}\sigma n)f(x-\sigma n,v_{-\theta})f(x,w_{-\theta})\la v_{-\theta}-w_{-\theta},n\ra_+^2\\\nonumber
    &\qquad\times\la v_{-\theta}+w_{-\theta},n\ra w\d n\d w\d \theta,
\end{align}
while
\begin{multline} \label{definition of energy current in x for Enskog collision integral}
    \I_4^E[f,f]=\frac{1}{2} \int_{\R^3\times \S^2}\int^\sigma_0\chi\left(x+(\frac{1}{2}\sigma-s) n\right)f(x-sn,v)f(x+(\sigma-s) n,w)\\
    \qquad\times\la v-w,n\ra_+^2\la v+w,n\ra n\d w\d n\d s,
\end{multline}
\end{theorem}
For Povzner collision integral, we have the following parallel result
\begin{theorem}\label{thm:Povzner}
For each rapidly decaying function $f\in C(\R^3\times \R^3)$, we have 
\begin{align}
     &P[f,f]=\grad_v\cdot \J_0^P[f,f]\label{Mass current for Povzner collision integral}\\
     &P[f,f]v_1=\grad_v\cdot \J_1^P[f,f]+\grad_x\cdot \I_1^P[f,f]\label{Momentum current 1 for Povzner collision integral}\\
     &P[f,f]v_2=\grad_v\cdot \J_2^P[f,f]+\grad_x\cdot \I_2^P[f,f]\label{Momentum current 2 for Povzner collision integral}\\
     &P[f,f]v_3=\grad_v\cdot \J_3^P[f,f]+\grad_x\cdot \I_3^P[f,f]\label{Momentum current 3 for Povzner collision integral}\\
     &P[f,f]|v|^2=\grad_v\cdot \J_4^P[f,f]+\grad_x\cdot \I_4^P[f,f]\label{Energy current for Povzner collision integral}
\end{align}
where vector fields$\I_k^P(x,v),\J_l^P(x,v)$, $k=1,2,3,4$, $l=0,1,2,3,4$ are defined by
\begin{equation}\label{definition of mass current for Povzner collision integral}
    \J_0^P=\int_{\R^3}\Jj_0^P[f,f](x,y,v)\d y,
\end{equation}
for $k=1,2,3$:
\begin{align}\label{definition of momentum current in v for Povzner collision integral}
    &\J_0^P[v_if,f]\\\nonumber
    &-\int_{\R^3}\Jj_0^P\left[\lla v,\nyx\rra \left(\nyx\right)_if,f\right]\d y\\\nonumber
    &+\int_{\R^3}\Jj_0^P\left[f,\lla w,\nyx\rra \left(\nyx\right)_if\right]\d y\\\nonumber
    &-\frac{1}{2}\int_{(\R^3)^4}\int^{\pi/2}_0f(x_{-\theta},v_{-\theta})f(y_{-\theta},w_{-\theta})J(x_{-\theta}-y_{-\theta},v_{-\theta}-w_{-\theta})  \\\nonumber
    &\qquad\qquad\times \lla v-w,\nyx\rra\left(\frac{x_{-\theta}-y_{-\theta}}{|x-y|}\right)_iw\d n\d w\d \theta,
\end{align}
and $l=1,2,3$
\begin{multline} \label{definition of momentum current in x for Povzner collision integral}
    \I_l^P[f,f]=-\frac{1}{2}\int_{(\R^3)^4}\int^{\pi/2}_0f(x_{-\theta},v_{-\theta})f(y_{-\theta},w_{-\theta})J(x_{-\theta}-y_{-\theta},v_{-\theta}-w_{-\theta})\\
    \qquad\qquad\times \lla v-w,\nyx\rra\left(\frac{x_{-\theta}-y_{-\theta}}{|x-y|}\right)_iy\d n\d w\d \theta,
\end{multline}
finally,
\begin{align}\label{definition of energy current in v for Povzner collision integral}
    &\J_4^P[f,f]\\\nonumber
    =&\J_0^P[|v|^2f,f]-\int_{\R^3}\Jj_0^P[\lla v,\nyx\rra^2f,f]\d y+\int_{\R^3}\Jj_0^P[f,\lla w,\nyx\rra^2f]\d y\\\nonumber
    &-\frac{1}{2}\int_{(\R^3)^2}\int^{\pi/2}_0f(x_{-\theta},v_{-\theta})f(y_{-\theta},w_{-\theta})J(x_{-\theta}-y_{-\theta},v_{-\theta}-w_{-\theta})\\
    &\qquad\times \lla v-w,\nyx\rra \lla v+w,\nyx\rra w\d y\d w\d \theta,
\end{align}
while
\begin{multline} \label{definition of energy current in x for Povzner collision integral}
    \I_4^P[f,f]= -\frac{1}{2}\int_{(\R^3)^2}\int^{\pi/2}_0f(x_{-\theta},v_{-\theta})f(y_{-\theta},w_{-\theta})J(x_{-\theta}-y_{-\theta},v_{-\theta}-w_{-\theta})\\
    \times \lla v-w,\nyx\rra \lla v+w,\nyx\rra y\d y\d w\d \theta,
\end{multline}
\end{theorem}
\section{Application to local conservation laws and H theorem}\label{sec:Application}
 To build up the connection with macroscopic equation, we need to define some macroscopic quantities
\begin{equation*}
  \begin{cases}
    \rho_f(t,x) := \int_{\mathbb{R}^3} f(t,x,v)\, \mathrm{d}v \\[1.5ex]
    u_f(t,x) := \frac{\mathbb{1}_{\rho_f>0}}{\rho_f(t,x)} \int_{\mathbb{R}^3} v f(t,x,v)\, \mathrm{d}v \\[1.5ex]
    P_f(t,x) := \int_{\mathbb{R}^3} \bigl(v-u_f(t,x)\bigr)^{\otimes 2} f(t,x,v)\, \mathrm{d}v \\[1.5ex]
    q_f(t,x) := \tfrac{1}{2} \int_{\mathbb{R}^3} \bigl(v-u_f(t,x)\bigr)\, |v-u_f(t,x)|^2 f(t,x,v)\, \mathrm{d}v,
  \end{cases}
\end{equation*}

where $\rho_f,u_f,P_f$, and $q_f$ represent the mass density, bulk velocity, stress tensor, and  heat-flow of the gas. If $f$ is a classical solution of the Boltzmann equation \eqref{genral kinetic equation} with the collision integral given by \eqref{boltzmann collision operator}, the following macroscopic equations for dilute gases can be derived:
\begin{equation}\label{macroscopic equations for dilute gas}
\begin{cases}
\p_t \rho_f + \grad_x \cdot (\rho_f u_f) = 0, \\
\p_t (\rho_f u_f) + \grad_x \cdot (\rho_f u_f^{\otimes 2}) + \grad_x \cdot (P_f) = 0, \\
\p_t \left(\frac{1}{2} |u_f|^2 + \mathrm{tr} P_f\right) + \grad_x \cdot \left[\left(\frac{1}{2} |u_f|^2 + \mathrm{tr} P_f\right) u_f + P_f \cdot u_f + q_f\right] = 0.
\end{cases}
\end{equation}
However, the situation changes significantly when dealing with the Enskog equation. 
\begin{corollary}
Let $f^E$ be a classical solution of the kinetic equation \eqref{genral kinetic equation} with collision integral \eqref{Standard Enskog collision integral}. Moreover $f^E\in C(\R^3_x\times\R^3_v)$ rapidly decaying with respect to $x$ and $v$. Then 
\begin{equation}\label{macroscopic equations for dense gas}
    \begin{cases}
    \p_t\rho_{f^E}+\grad_x\cdot (\rho_{f^E}u_{f^E})=0\\
    \p_t(\rho_{f^E}u_{f^E})+\grad_x\cdot(\rho_{f^E}u_{f^E}^{\otimes2})+\grad_x\cdot(P_{f^E}-\mathbb{I}^E)=0\\
    \p_t(\frac{1}{2}|u_{f^E}|^2+\mathrm{tr}P_{f^E})\\\qquad+\grad_x\cdot[(\frac{1}{2}|u_{f^E}|^2+\mathrm{tr}P_{f^E})u_{f^E}+P_{f^E}\cdot u_{f^E}+q_{f^E}-\I_4^E]=0,
    \end{cases}
\end{equation}
where $\I^E_i, i=1,2,3$ are components of $\mathbb{I}^E$, namely
\[
\mathbb{I}^E(t,x,v)=\begin{pmatrix}
\I_1^E\\\I_2^E\\\I_3^E
\end{pmatrix}
\]
Correspondingly, we have the similar result for the Povzner equation.
\end{corollary}
To conclude the corollary, we only need to integrate \eqref{Mass current for Enskog collision integral}–\eqref{Energy current for Enskog collision integral} with respect to $v$ over $\mathbb{R}^3$. However, it is important to note that the equations in \eqref{macroscopic equations for dense gas} are not closed. Closure is achieved only after solving the general kinetic equation \eqref{genral kinetic equation} and computing $\mathcal{I}^E_4$, $\mathbb{I}^E$, and $q_{f^E}$. Furthermore, a closure relation is necessary. For example, to derive the Euler system, the distribution function $f$ should approximate a Maxwellian such that $q_{f^E}$ and the anisotropic part of $P_{f^E}$ become negligible, i.e., 

\[
q_{f^E} = \mathbf{0}, \qquad P_{f^E} = p \mathrm{Id}_3.
\]

This reduction allows us to solve the equations with four unknowns ($\rho, u, p$), instead of handling ten unknowns with only four constraints. 

The corollary not only clarifies the influence of molecular volume occupation on macroscopic equations but also highlights a key distinction: the Boltzmann collision integral does not contribute to the stress tensor $P_f$—even a free transport equation can generate a stress tensor. In contrast, the Enskog and Povzner collsion integrals have a direct impact on the pressure tensor. This distinction is significant because the coefficient of pressure depends on the adiabatic index, indicating that the gas governed by equations \eqref{macroscopic equations for dense gas} no longer behaves as a perfect gas. This observation provides an alternative perspective on the limitations of the perfect gas assumption for such systems.

It is straightforward to establish the local entropy inequality, as demonstrated in \cite{CCG2024}, due to the availability of a stronger weak formulation for the Povzner collision integral:
\begin{multline*}
(P[f,f], \phi)_{x,v} = \frac{1}{2} \int_{(\R^3)^4} f(x,v) f(y,w) J(x-y, v-w)\\
\times\big[\phi(x,v') + \phi(y,w') - \phi(x,v) - \phi(y,w)\big]  \d x \d y \d v \d w,
\end{multline*}
as given in equation (10) of \cite{povzner1965}. If we choose $\phi(x,v) = \ln f(x,v)$, the following entropy inequality is obtained:

\begin{multline*}
\dt \int_{\R^3 \times \R^3} f(t,x,v) \ln f(t,x,v) \d x \d v \\
= \frac{1}{2} \int_{(\R^3)^4} f(x,v) f(y,w) J(x-y, v-w) \ln\left(\frac{f(x,v') f(y,w')}{f(x,v) f(y,w)}\right)  \d x \d y \d v \d w \\\leq 0,
\end{multline*}
which follows from the basic inequality: for $a,b>0$, \[a \ln \frac{b}{a} \leq b - a.\]

However, this result is less remarkable for the Povzner equation because the entropy inequality is identical to the original Boltzmann's entropy. As such, we do not discuss it further here.
As for the Standard Enskog equation, the local and global entropy inequality are studied in \cite{Goldstein1997}.

\section{Proof of Theorem \ref{thm:Enskog}}\label{sec:proof of Enskog}
The following proof is based on multiplying the collision kernel by a test function. Before we prove the theorem, we begin with a lemma which will be used repeatedly.

\begin{lemma}\label{lem:Enskog}
Let $f,g\in C(\R^3\times\R^3)$ with rapid decay at infinity. For each test function $\phi\in  C(R^3\times\R^3)$ with at most polynomial growth at infinity,
\begin{multline}
    (E[f,g],\phi)_{x,v}:= \int_{\R^3\times\R^3}E[f,g](x,v)\phi(x,v)\d x\d v\\
     =\int_{(\R^3)^4}\chi(x+\frac{1}{2}\sigma n)f(x,v)g(x+\sigma n,w)\la v-w,n\ra_+\\
     \times[\phi(x,v')-\phi(x,v)]\d x\d n\d v\d w
\end{multline}

\end{lemma}
\begin{proof}
For the gain term of the Standard Enskog collision integral defined in \ref{Standard Enskog collision integral}, we apply the transform $(v,w,n)\mapsto(v',w',-n)$. Since the absolute value of the Jacobian determinate is $1$ \eqref{Jacobian of collision}, the gain term becomes
\begin{multline*}
    \int_{(\R^3)^4}\chi(x+\frac{1}{2}\sigma n)f(x,v''(v',w',-n))g(x+\sigma n,w''(v',w',-n))\la v'-w',-n\ra_+\\
    \times\phi(x,v'(v,w,-n))\d x\d n\d v\d w.
\end{multline*}
The properties \eqref{reversible collision}, \eqref{post-collision inverses the direction} and \eqref{collision is even wrsp n} show that this term is equal to
\[
\int_{(\R^3)^4}\chi(x+\frac{1}{2}\sigma n)f(x,v)g(x+\sigma n,w)\la v-w,n\ra_+\phi(x,v'(v,w,n))\d x\d n\d v\d w.
\]
Summing with the loss term leads to the sought result.
\end{proof}
For brevity, $\big(A(x,v),b(x,v)\big)_{x,v}$ is designated to $\int_{\R^3\times\R^3}\la A(x,v),b(x,v)\ra\d x\d v$. 
Indeed, in the standard Enskog equation, the correlation function $\chi$ depends on the distribution function $f$, whereas the following proof does not rely on this dependence. For clarity, throughout the proof, whenever we introduce the bilinear operator for the Enskog collision integral—for instance, $\Jj^E_i[\cdot,\cdot]$ and $\J^E_i[\cdot,\cdot]$—we assume that the function $\chi$ is independent of its functional argument.
\subsection{Mass current for Enskog collision integral: Proof of \eqref{Mass current for Enskog collision integral} }
\begin{proof}
As the first step, we apply the Lemma \ref{lem:Enskog}, then note the following identity
\begin{multline}\label{pf:mass current}
    \phi(x,v')-\phi(x,v)=\int_{0}^{\la v-w,n\ra}\frac{\d}{\d s}\phi(x,v-sn)\d s\\
     =-\int_{0}^{\la v-w,n\ra}\la n,\grad_v\phi(x,v-sn)\ra\d s,
\end{multline}
hence we can write $(E[f,g],\phi)_{x,v}$ as 
\begin{multline*}
-\int_{(\R^3)^3\times\S^2}\int_{0}^{\la v-w,n\ra}\chi(x+\frac{1}{2}\sigma n)f(x,v)g(x+\sigma n,w)\la v-w,n\ra_+ \\
\times\la n,\grad_v\phi(x,v-sn)\ra\d s\d n\d w\d x\d v.
\end{multline*}
Subsequently, we translate $v$ and $w$ along $n$ by $s$, namely
\begin{multline*}
-\int_{(\R^3)^3\times\S^2}\int_{0}^{\la v-w,n\ra}\chi(x+\frac{1}{2}\sigma n)f(x,v+sn)g(x+\sigma n,w+sn)\la v-w,n\ra_+ \\
\times\la n,\grad_v\phi(x,v)\ra\d s\d n\d w\d x\d v.
\end{multline*}
With the notation of the Landau mass current of the Enskog collision integral \eqref{Landau current for the Standard Enskog collision integral}, we may represent $(E[f,g],\phi)_{x,v}$ as
\[
\left(-\int_{\S^2}\Jj_0^E[f,g]\d n,\grad_v\phi(x,v)\right)_{x,v}
\]
We merely need to let $f=g$ to deduce the result in the sense of distributions.
\end{proof}
In fact, we established a stronger result than \eqref{Mass current for Enskog collision integral}, allowing for different distribution functions in the collision integral. We will utilize this stronger result later to express the momentum and energy current.
\subsection{Momentum current for Enskog collision integral: Proof of \eqref{Momentum current 1 for Enskog collision integral}\eqref{Momentum current 2 for Enskog collision integral}\eqref{Momentum current 3 for Enskog collision integral}}
\begin{proof}
Again, we apply Lemma \ref{lem:Enskog} to test function $v_i\phi(x,v)$ for $i=1,2,3$. We reformulate $(E[f,f]v_i,\phi)$ as
\begin{multline*}
    \int_{(\R^3)^3\times\S^2}\chi(x+\frac{1}{2}\sigma n)f(x,v)f(x+\sigma n,w)\la v-w,n\ra_+\\
    \times \big[v'_i\phi(x,v')-v_i\phi(x,v)\big]\d x\d v\d n\d w.
\end{multline*}
Then we split the difference $v'\phi(x,v')-v\phi(x,v)$ into two parts,
\[
v_i'[\phi(x,v')-\phi(x,v)]+[v_i'-v_i]\phi(x,v)
\]
Thereafter, we use the $i$-th component of \eqref{pre-to-post collision transformation} to replace $v_i'$, hence
\begin{multline*}
\underbrace{v_i[\phi(x,v')-\phi(x,v)]}_{=:\phi_1}-\underbrace{\la v,n\ra n_i[\phi(x,v')-\phi(x,v)]}_{=:\phi_2}\\
+\underbrace{\la w,n\ra n_i[\phi(x,v')-\phi(x,v)]}_{=:\phi_3}-\underbrace{\la v-w,n\ra n_i\phi(x,v)}_{=:\phi_4}.
\end{multline*}
We use Lemma \ref{lem:Enskog} for $\phi_1$, then it shows
\begin{multline*}
 \int_{(\R^3)^3\times\S^2}\chi(x+\frac{1}{2}\sigma n)f(x,v)f(x+\sigma n,w)\la v-w,n\ra_+\phi_1\d x\d v\d n\d w\\
 =(E[v_if,f],\phi)_{x,v}.
\end{multline*}
Following this, we use the result \eqref{Mass current for Enskog collision integral} obtained in the last section. It shows that
\begin{multline*}
 \int_{(\R^3)^\times\S^2}\chi(x+\frac{1}{2}\sigma n)f(x,v)f(x+\sigma n,w)\la v-w,n\ra_+\phi_1\d x\d v\d n\d w\\
 =(-\J_0^E[v_if,f],\grad_v\phi)_{x,v}.
\end{multline*}
By the same reasoning, we can deal with $\phi_2$ and $\phi_3$,
\begin{multline*}
 \int_{(\R^3)^3\times\S^2}\chi(x+\frac{1}{2}\sigma n)f(x,v)f(x+\sigma n,w)\la v-w,n\ra_+\phi_2\d x\d v\d n\d w\\
 =\left(-\int_{\S^2}\Jj_0^E[\la v,n\ra n_if,f]\d n,\grad_v\phi\right)_{x,v},
\end{multline*}
and
\begin{multline*}
 \int_{(\R^3)^3\times\S^2}\chi(x+\frac{1}{2}\sigma n)f(x,v)f(x+\sigma n,w)\la v-w,n\ra_+\phi_3\d x\d v\d n\d w\\
 =\left(-\int_{\S^2}\Jj_0^E[f,\la w,n\ra n_if]\d n,\grad_v\phi\right)_{x,v}.
\end{multline*}
Regarding $\phi_4$, we change of variables from $(v,w,x,n)$ to $(w,v,x+\sigma n.-n)$,
\begin{multline*}
 \int_{(\R^3)^3\times\S^2}\chi(x+\frac{1}{2}\sigma n)f(x,v)f(x+\sigma n,w)\la v-w,n\ra_+\phi_4\d x\d v\d n\d w\\
 = -\int_{(\R^3)^3\times\S^2}\chi(x+\frac{1}{2}\sigma n)f(x,v)f(x+\sigma n,w)\la v-w,n\ra_+^2n_i\phi(x+\sigma n,w)\\\d x\d v\d n\d w.
\end{multline*}
Therefore, we symmetrize it,
\begin{multline*}
 \int_{(\R^3)^3\times\S^2}\chi(x+\frac{1}{2}\sigma n)f(x,v)f(x+\sigma n,w)\la v-w,n\ra_+\phi_4\d x\d v\d n\d w\\
 = \frac{1}{2}\int_{(\R^3)^3\times\S^2}\chi(x+\frac{1}{2}\sigma n)f(x,v)f(x+\sigma n,w)\la v-w,n\ra_+^2n_i\\
 \times \big[\phi(x,v)-\phi(x+\sigma n,w)\big]\d x\d v\d n\d w.
\end{multline*}
Subsequently, we decompose $\phi(x,v)-\phi(x+\sigma n, w)$ into
\[
\underbrace{\phi(x,v)-\phi(x+\sigma n, v)}_{=:\phi_{41}}+\underbrace{\phi(x+\sigma n,v)-\phi(x+\sigma n, w)}_{=:\phi_{42}}.
\]
Consider the following two identities
\begin{equation}\label{pf translation in x Enskog}
\phi(x,v)-\phi(x+\sigma n, v)=\int_{0}^\sigma\frac{\d}{\d s}\phi(x+sn,v)\d s=\int_0^\sigma \la n,\grad_x\phi(x+sn,v)\ra\d s,
\end{equation}
and
\begin{multline}\label{pf rotation in v Enskog}
\phi(x+\sigma n,v)-\phi(x+\sigma n, w)=-\int_{0}^{\pi/2}\frac{\d}{\d \theta}\phi(x+\sigma n,v_\theta)\d \theta\\
=-\int_0^{\pi/2} \la w_\theta,\grad_v\phi(x+\sigma n,v_\theta)\ra\d \theta,    
\end{multline}
where $(v_\theta,w_\theta)$ are defined by
\begin{equation}\label{definition of rotation in velocities}
    \begin{pmatrix}
    v_\theta\\
    w_\theta
    \end{pmatrix}
    :=\begin{pmatrix}
    \cos \theta & \sin \theta\\
    -\sin \theta & \cos \theta
    \end{pmatrix}
    \begin{pmatrix}
    v\\
    w
    \end{pmatrix}.
\end{equation}
 since 
\[
\frac{\d}{\d\theta}\phi(x,v_\theta)=\la w_\theta,\grad_v\phi(x,v_\theta)\ra.
\]
The reason for choosing these variables is to ensure that the determinant of the Jacobian equals $1$, namely,
\[
\det\left(\frac{\p(v_\theta,w_\theta)}{\p(v,w)}\right)=1.
\]
Hence we plug \eqref{pf translation in x Enskog} into $\phi_{41}$ term,
\begin{multline*}
    \int_{(\R^3)^3\times\S^2}\chi(x+\frac{1}{2}\sigma n)f(x,v)f(x+\sigma n,w)\la v-w,n\ra_+^2n_i\phi_{41}\d x\d v\d n\d w\\
    = \int_{(\R^3)^3\times\S^2}\int^\sigma_0\chi(x+\frac{1}{2}\sigma n)f(x,v)f(x+\sigma n,w)\la v-w,n\ra_+^2n_i\\
    \times\la n,\grad_x\phi(x+sn,v)\ra\d x\d v\d n\d w\d s.
\end{multline*}
Then we replace $x$ by $x-sn$, it becomes
\begin{multline*}
    \int_{(\R^3)^3\times\S^2}\chi(x+\frac{1}{2}\sigma n)f(x,v)f(x+\sigma n,w)\la v-w,n\ra_+^2n_i\phi_{41}\d x\d v\d n\d w\\
    = \left(\int_{\R^3\times \S^2}\int^\sigma_0\chi(x+(\frac{1}{2}\sigma-s) n)f(x-sn,v)f(x+(\sigma-s) n,w)\right.\\
    \times\la v-w,n\ra_+^2n_in\d w\d n\d s,\grad_x\phi(x,v)\bigg)_{x,v}.
\end{multline*}
In an analogous fashion, we can substitute \eqref{pf rotation in v Enskog} into $\phi_{42}$ term,
\begin{multline*}
    \int_{(\R^3)^3\times\S^2}\chi(x+\frac{1}{2}\sigma n)f(x,v)f(x+\sigma n,w)\la v-w,n\ra_+^2n_i\phi_{42}\d x\d v\d n\d w\\
    = - \int_{(\R^3)^3\times\S^2}\int^{\pi/2}_0\chi(x+\frac{1}{2}\sigma n)f(x,v)f(x+\sigma n,w)\la v-w,n\ra_+^2n_i\\
    \times\la w_\theta,\grad_v\phi(x+\sigma n,v_\theta)\ra\d x\d v\d n\d w\d \theta.
\end{multline*}
Then we rotate the velocities by $-\theta$, and translate $x$ along $n$ by $\sigma$,
\begin{multline*}
    \int_{(\R^3)^3\times\S^2}\chi(x+\frac{1}{2}\sigma n)f(x,v)f(x+\sigma n,w)\la v-w,n\ra_+^2n_i\phi_{42}\d x\d v\d n\d w\\
    = \bigg(-\int_{(\R^3)^3\times\S^2}\int^{\pi/2}_0\chi(x-\frac{1}{2}\sigma n)f(x-\sigma n,v_{-\theta})f(x,w_{-\theta})\la v_{-\theta}-w_{-\theta},n\ra_+^2n_i\\
    \times w\d n\d w\d \theta,\grad_v\phi(x,v)\bigg)_{x,v}.
\end{multline*}
For $i=1,2,3 $, sum them all together, we arrive the weak formulation of \eqref{Momentum current 1 for Enskog collision integral}-\eqref{Momentum current 3 for Enskog collision integral}:
\begin{align*}
    &\left(v_iE[f,f](x,v),\phi(x,v)\right)_{x,v}=(-\J_0^E[v_if,f],\grad_v\phi)_{x,v}-\left(-\int_{\S^2}\Jj_0^E[\la v,n\ra n_if,f]\d n,\grad_v\phi\right)_{x,v}\\
    &+\left(-\int_{\S^2}\Jj_0^E[f,\la w,n\ra n_if]\d n,\grad_v\phi\right)_{x,v}\\
    &-\frac{1}{2} \left(\int_{\R^3\times \S^2}\int^\sigma_0\chi(x+(\frac{1}{2}\sigma-s) n)f(x-sn,v)f(x+(\sigma-s) n,w)\right.\\
    &\qquad\times\la v-w,n\ra_+^2n_in\d w\d n\d s,\grad_x\phi(x,v)\bigg)_{x,v}\\
    &-\frac{1}{2} \bigg(-\int_{(\R^3)^4}\int^{\pi/2}_0\chi(x-\frac{1}{2}\sigma n)f(x-\sigma n,v_{-\theta})f(x,w_{-\theta})\la v_{-\theta}-w_{-\theta},n\ra_+^2\\
    &\qquad n_i w\d n\d w\d \theta,\grad_v\phi(x,v)\bigg)_{x,v}.
\end{align*}
\end{proof}

\subsection{Energy current for Enskog collision integral: Proof of \eqref{Energy current for Enskog collision integral}}
\begin{proof}
Substituting $\phi(x,v)=|v|^2\varphi(x,v)$ into Lemma \ref{lem:Enskog}, we obtain:
\begin{multline*}
    \int_{(\R^3)^3\times\S^2}\chi(x+\frac{1}{2}\sigma n)f(x,v)f(x+\sigma n,w)\la v-w,n\ra_+\\
    \times \big[|v'|^2\varphi(x,v')-|v|^2\varphi(x,v)\big]\d x\d v\d n\d w.
\end{multline*}
Then, with the help of \eqref{pre-to-post collision transformation}, we split the difference $|v'|^2\varphi(x,v')-|v|^2\varphi(x,v)$ into four parts,
\begin{multline*}
\underbrace{|v|^2[\varphi(x,v')-\varphi(x,v)]}_{=:\varphi_1}-\underbrace{\la v,n\ra^2[\phi(x,v')-\varphi(x,v)]}_{=:\varphi_2}\\
+\underbrace{\la w,n\ra^2[\varphi(x,v')-\varphi(x,v)]}_{=:\varphi_3}-\underbrace{(\la v,n\ra ^2-\la w,n\ra^2)\varphi(x,v)}_{=:\varphi_4},
\end{multline*}
since
\[
|v'|^2=|v-\la v-w,n\ra n|^2=|v|^2-\la v,n\ra^2+\la w,n\ra^2.
\]

Following this, we use the result \eqref{Mass current for Enskog collision integral}, which we obtained in the last section. We find that
\begin{multline*}
 \int_{(\R^3)^3\times\S^2}\chi(x+\frac{1}{2}\sigma n)f(x,v)f(x+\sigma n,w)\la v-w,n\ra_+(\varphi_1-\varphi_2+\varphi_3)\d x\d v\d n\d w\\
 =\left(-\int_{\S^2}\Jj_0^E[(|v|^2-\la v,n\ra^2)f,f]+\Jj_0^E[f,\la w,n\ra^2f]\d n,\grad_v\phi\right)_{x,v}.
\end{multline*}
By the same process showed as above, we perform the change of variables $(v,w,x,n)\mapsto (w,v,x+\sigma n.-n)$ in the $\varphi_4$ term, then symmetrize it,

\begin{multline*}
 \int_{(\R^3)^3\times\S^2}\chi(x+\frac{1}{2}\sigma n)f(x,v)f(x+\sigma n,w)\la v-w,n\ra_+\varphi_4\d x\d v\d n\d w\\
 = \frac{1}{2}\int_{(\R^3)^3\times\S^2}\chi(x+\frac{1}{2}\sigma n)f(x,v)f(x+\sigma n,w)\la v-w,n\ra_+^2\la v+w,n\ra\\
 \times \big[\varphi(x,v)-\varphi(x+\sigma n,w)\big]\d x\d v\d n\d w.
\end{multline*}
Therefore, we replace $\varphi(x,v)-\varphi(x+\sigma n,w)$ by \eqref{pf rotation in v Enskog} and \eqref{pf translation in x Enskog}
\begin{multline*}
    \int_{(\R^3)^3\times\S^2}\chi(x+\frac{1}{2}\sigma n)f(x,v)f(x+\sigma n,w)\la v-w,n\ra_+^2n_i\varphi_4\d x\d v\d n\d w\\
    = \int_{(\R^3)^3\times\S^2}\int^\sigma_0\chi(x+\frac{1}{2}\sigma n)f(x,v)f(x+\sigma n,w)\la v-w,n\ra_+^2\la v+w,n\ra\\
    \times\la n,\grad_x\phi(x+sn,v)\ra\d x\d v\d n\d w\d s\\
  - \int_{(\R^3)^3\times\S^2}\int^{\pi/2}_0\chi(x+\frac{1}{2}\sigma n)f(x,v)f(x+\sigma n,w)\la v-w,n\ra_+^2\la v+w,n\ra\\
    \times\la w_\theta,\grad_v\phi(x+\sigma n,v_\theta)\ra\d x\d v\d n\d w\d \theta.
\end{multline*}
Finally, we perform the change of variables to ensure the variables of $\varphi$ is exact $x$ and $v$,
\begin{multline*}
    \int_{(\R^3)^3\times\S^2}\chi(x+\frac{1}{2}\sigma n)f(x,v)f(x+\sigma n,w)\la v-w,n\ra_+^2\la v+w,n\ra\varphi_{4}\d x\d v\d n\d w\\
    = - \int_{(\R^3)^3\times\S^2}\int^{\pi/2}_0\chi(x+\frac{1}{2}\sigma n)f(x,v)f(x+\sigma n,w)\la v-w,n\ra_+^2\la v+w,n\ra\\
    \times\la w_\theta,\grad_v\varphi(x+\sigma n,v_\theta)\ra\d x\d v\d n\d w\d \theta.
\end{multline*}
Then we rotate the velocities by $-\theta$, and translate $x$ along $n$ by $\sigma$,
\begin{multline*}
    \int_{(\R^3)^3\times\S^2}\chi(x+\frac{1}{2}\sigma n)f(x,v)f(x+\sigma n,w)\la v-w,n\ra_+^2\la v+w,n\ra\varphi_{4}\d x\d v\d n\d w\\
    = \bigg(-\int_{(\R^3)^3\times\S^2}\int^{\pi/2}_0\chi(x-\frac{1}{2}\sigma n)f(x-\sigma n,v_{-\theta})f(x,w_{-\theta})\la v_{-\theta}-w_{-\theta},n\ra_+^2\\
    \times \la v_{-\theta}+w_{-\theta},n\ra w\d n\d w\d \theta,\grad_v\varphi(x,v)\bigg)_{x,v}.
\end{multline*}

In summary, we have
\begin{align*}
    &\left(|v|^2E[f,f](x,v),\varphi(x,v)\right)_{x,v}=(-\J_0^E[|v|^2f,f],\grad_v\phi)_{x,v}\\
    &-\left(-\int_{\S^2}\Jj_0^E[\la v,n\ra^2f,f]\d n,\grad_v\varphi\right)_{x,v}+\left(-\int_{\S^2}\Jj_0^E[f,\la w,n\ra^2f]\d n,\grad_v\varphi\right)_{x,v}\\
    &-\frac{1}{2} \left(\int_{\R^3\times \S^2}\int^\sigma_0\chi(x+(\frac{1}{2}\sigma-s) n)f(x-sn,v)f(x+(\sigma-s) n,w)\right.\\
    &\qquad\times\la v-w,n\ra_+^2\la v+w,n\ra n\d w\d n\d s,\grad_x\varphi(x,v)\bigg)_{x,v}\\
    &-\frac{1}{2} \bigg(-\int_{\R^3\times\S^2}\int^{\pi/2}_0\chi(x-\frac{1}{2}\sigma n)f(x-\sigma n,v_{-\theta})f(x,w_{-\theta})\la v_{-\theta}-w_{-\theta},n\ra_+^2\\
    &\qquad\times\la v_{-\theta}+w_{-\theta},n\ra w\d n\d w\d \theta,\grad_v\varphi(x,v)\bigg)_{x,v},
\end{align*}
which is the weak formulation of \eqref{Energy current for Enskog collision integral}.
\end{proof}

\section{Proof of Theorem \ref{thm:Povzner}}\label{sec:proof of Povzner}
We have analogous results for Povzner collision integral.
\begin{lemma}\label{lem:Povzner}
Let $f,g\in C(\R^3\times\R^3)$ with rapidly decay at infinity. For each test function $\phi\in  C(R^3\times\R^3)$ with at most polynomial growth at infinity,
\begin{multline}\label{pf lem Povzner}
     (P[f,g],\phi)_{x,v}:=\int_{\R^3\times\R^3}P[f,g](x,v)\phi(x,v)\d x\d v\\
     =\int_{(\R^3)^4}f(x,v)g(y,w)J(x-y,v-w)[\phi(x,v')-\phi(x,v)]\d x\d y\d v\d w.
\end{multline}
\end{lemma}
\begin{proof}
We do the transformation $(v,w)\mapsto (v',w')$ for the gain term in \eqref{Povzner collision integral}, then the left-hand side of \eqref{pf lem Povzner} becomes
\[
\int_{(\R^3)^4}f(x,v)g(y,w)J(x-y,v'-w')\phi(x,v')\d x\d y\d v\d w.
\]
We can conclude with the help of assumption \eqref{Povzner kernel is invariant after collision}.
\end{proof}
\subsection{Mass current for Povzner collision integral: Proof of \eqref{Mass current for Povzner collision integral} }
\begin{proof}
Note that
\begin{multline*}
    \phi(x,v')-\phi(x,v)=\int_{0}^{\lla v-w,\nyx\rra}\frac{\d}{\d s}\phi\left(x,v-s\nyx\right)\d s\\
     =-\int_{0}^{\left\la v-w,\nyx\right\ra}\lla \nyx,\grad_v\phi(x,v-s\nyx)\rra\d s,
\end{multline*}
We substitute the identity above into \eqref{lem:Povzner},
Subsequently, we translate $v$ and $w$ in the direction $\nyx$ by $s$, then we can rewrite $(P[f,g],\phi)_{x,v}$ by
\begin{multline*}
-\int_{(\R^3)^4}\int_{0}^{\lla v-w,\nyx\rra}f\left(x,v+s\nyx\right)g\left(y,w+s\nyx\right) \\
\times J\left(x-y,v-w\right)\lla \nyx,\grad_v\phi(x,v)\rra\d s\d y\d w\d x\d v.
\end{multline*}
Using the notation of Landau mass current for the Povzner collision integral \eqref{Landau current for the Povzner collision integral}, we may represent $(P[f,g],\phi)_{x,v}$ by
\[
\left(-\int_{\R^3}\Jj_0^P[f,g]\d y,\grad_v\phi(x,v)\right)_{x,v}
\]
We merely need to let $f=g$ to deduce the desired result in the sense of distributions.

\end{proof}

\subsection{Momentum current for Povzner collision integral: Proof of \eqref{Momentum current 1 for Povzner collision integral}\eqref{Momentum current 2 for Povzner collision integral}\eqref{Momentum current 3 for Povzner collision integral}}
\begin{proof}
We apply Lemma \ref{lem:Povzner} to the test function $v_i\phi(x,v)$ for $i=1,2,3$. We reformulate $(P[f,f]v_i,\phi)$ by
\begin{multline*}
    \int_{(\R^3)^4}f(x,v)f\left(y,w\right)J(x-y,v-w)\big[v'_i\phi(x,v')-v_i\phi(x,v)\big]\d x\d v\d y\d w.
\end{multline*}
Then we split the difference $v'\phi(x,v')-v\phi(x,v)$ into two parts,
\[
v_i'[\phi(x,v')-\phi(x,v)]+[v_i'-v_i]\phi(x,v)
\]
Thereafter, we use the $i$-th component of \eqref{pre-to-post collision transformation} to replace $v_i'$, hence
\begin{multline*}
\underbrace{v_i[\phi(x,v')-\phi(x,v)]}_{=:\phi_1}-\underbrace{\lla v,\nyx\rra \left(\nyx\right)_i[\phi(x,v')-\phi(x,v)]}_{=:\phi_2}\\
+\underbrace{\la w,\nyx\ra \left(\nyx\right)_i[\phi(x,v')-\phi(x,v)]}_{=:\phi_3}\\
-\underbrace{\lla v-w,\nyx\rra \left(\nyx\right)_i\phi(x,v)}_{=:\phi_4}.
\end{multline*}
We use Lemma \ref{lem:Povzner} for $\phi_1$, which shows
\begin{multline*}
 \int_{(\R^3)^4}f(x,v)f(y,w)\jxy\phi_1\d x\d v\d y\d w=(P[v_if,f],\phi)_{x,v}.
\end{multline*}
Following this, we use \eqref{Mass current for Povzner collision integral}, which we obtained in the last section. It shows that
\begin{multline*}
 \int_{(\R^3)^4}f(x,v)f(y,w)\jxy\phi_1\d x\d v\d y\d w\\
 =(-\J_0^P[v_if,f],\grad_v\phi)_{x,v}.
\end{multline*}
By the same token, we can deal with $\phi_2$ and $\phi_3$,
\begin{multline*}
 \int_{(\R^3)^4} f(x,v)f(y,w)\jxy\phi_2\d x\d v\d y\d w\\
 =\left(-\int_{\R^3}\Jj_0^P\left[\lla v,\nyx\rra \left(\nyx\right)_if,f\right]\d y,\grad_v\phi\right)_{x,v},
\end{multline*}
and
\begin{multline*}
 \int_{(\R^3)^4} f(x,v)f(x+\sigma n,w)\jxy\phi_3\d x\d v\d y\d w\\
 =\left(-\int_{\R^3}\Jj_0^P\left[f,\lla w,\nyx\rra \left(\nyx\right)_if\right]\d y,\grad_v\phi\right)_{x,v}.
\end{multline*}
Thanks to \eqref{Povzner collision kernel is time reversible}, we change variables from $(v,w,x,y)$ to $(w,v,y,x)$,
\begin{multline*}
 \int_{(\R^3)^4} f(x,v)f(y,w)\jxy\phi_4\d x\d v\d n\d w\\
 = -\int_{(\R^3)^4} f(x,v)f(y,w)\jxy \lla v-w,\nyx\rra\left(\nyx\right)_i\\
 \times\phi(y,w)\d x\d v\d n\d w,
\end{multline*}
Symmetrizing the identity above
\begin{multline*}
 \int_{(\R^3)^4} f(x,v)f(y,w)\jxy\phi_4\d x\d v\d y\d w\\
 = \frac{1}{2}\int_{(\R^3)^4} f(x,v)f(y,w)\jxy \lla v-w,\nyx\rra \left(\nyx\right)_i\\
 \times \big[\phi(x,v)-\phi(y,w)\big]\d x\d v\d y\d w.
\end{multline*}
Notice that
\begin{multline}\label{pf rotation in x and v Povzner}
\phi(x,v)-\phi(y, w)=-\int_{0}^{\pi/2}\frac{\d}{\d \theta}\phi(x_\theta,v_\theta)\d \theta\\
=-\int_0^{\pi/2} \la y_\theta,\grad_x\phi(x_\theta,v_\theta)\ra+\la w_\theta,\grad_v\phi(x_\theta,v_\theta)\ra\d \theta,    
\end{multline}
where $(v_\theta,w_\theta)$ are defined in \eqref{definition of rotation in velocities}, and $(x_\theta,y_\theta)$ are determined by
\begin{equation}
    \begin{pmatrix}
    x_\theta\\
    y_\theta
    \end{pmatrix}
    :=\begin{pmatrix}
    \cos \theta & \sin \theta\\
    -\sin \theta & \cos \theta
    \end{pmatrix}
    \begin{pmatrix}
    x\\
    y
    \end{pmatrix}.
\end{equation}

Hence we plug \eqref{pf rotation in x and v Povzner} into the $\phi_{4}$ term,
\begin{multline*}
    \int_{(\R^3)^4}f(x,v)f(y,w)\jxy \lla v-w,\nyx\rra \left(\nyx\right)_i\phi_{4}\d x\d v\d n\d w\\
    = \int_{(\R^3)^4}\int^{\pi/2}_0f(x,v)f(y,w)\jxy \lla v-w,\nyx\rra \left(\nyx\right)_i\\
    \times\left[\la y_\theta,\grad+\phi(x_\theta,v_\theta)\ra+\la w_\theta,\grad_v\phi(x_\theta,v_\theta)\ra\right]\d x\d v\d n\d w\d s.
\end{multline*}

Then we rotate both the positions and velocities by $-\theta$,
\begin{multline*}
    \int_{(\R^3)^4}f(x,v)f(y,w)\jxy \lla v-w,\nyx\rra \left(\nyx\right)_i\phi_{4}\d x\d v\d n\d w\\
    = \bigg(-\int_{(\R^3)^4}\int^{\pi/2}_0f(x_{-\theta},v_{-\theta})f(y_{-\theta},w_{-\theta})J(x_{-\theta}-y_{-\theta},v_{-\theta}-w_{-\theta}) \lla v-w,\nyx\rra \\
    \times \left(\frac{x_{-\theta}-y_{-\theta}}{|x-y|}\right)_iw\d n\d w\d \theta,\grad_v\phi(x,v)\bigg)_{x,v}\\
    +\bigg(-\int_{(\R^3)^4}\int^{\pi/2}_0f(x_{-\theta},v_{-\theta})f(y_{-\theta},w_{-\theta})J(x_{-\theta}-y_{-\theta},v_{-\theta}-w_{-\theta}) \lla v-w,\nyx\rra \\
    \times \left(\frac{x_{-\theta}-y_{-\theta}}{|x-y|}\right)_iy\d n\d w\d \theta,\grad_x\phi(x,v)\bigg)_{x,v},
\end{multline*}
since 
\[
 \lla v-w,\nyx\rra= \lla v_{-\theta}-w_{-\theta},\frac{y_{-\theta}-x_{-\theta}}{|y_{-\theta}-x_{-\theta}|}\rra,
\]
and 
\[
|y_{-\theta}-x_{-\theta}|=|y-x|.
\]
For $i=1,2,3 $, sum them all together, we arrive the weak formulation of \eqref{Momentum current 1 for Povzner collision integral}-\eqref{Momentum current 3 for Povzner collision integral}:
\begin{align*}
    &\left(v_iP[f,f](x,v),\phi(x,v)\right)_{x,v}=(-\J_0^P[v_if,f],\grad_v\phi)_{x,v}\\
    &\qquad\qquad\qquad-\left(-\int_{\R^3}\Jj_0^P\left[\lla v,\nyx\rra \left(\nyx\right)_if,f\right]\d y,\grad_v\phi\right)_{x,v}\\
    &\qquad\qquad\qquad+\left(-\int_{\R^3}\Jj_0^P\left[f,\lla w,\nyx\rra \left(\nyx\right)_if\right]\d y,\grad_v\phi\right)_{x,v}\\
    &-\frac{1}{2} \bigg(-\int_{(\R^3)^4}\int^{\pi/2}_0f(x_{-\theta},v_{-\theta})f(y_{-\theta},w_{-\theta})J(x_{-\theta}-y_{-\theta},v_{-\theta}-w_{-\theta})  \\
    &\qquad\qquad\times \lla v-w,\nyx\rra\left(\frac{x_{-\theta}-y_{-\theta}}{|x-y|}\right)_iw\d n\d w\d \theta,\grad_v\phi(x,v)\bigg)_{x,v}\\
    &-\frac{1}{2}\bigg(-\int_{(\R^3)^4}\int^{\pi/2}_0f(x_{-\theta},v_{-\theta})f(y_{-\theta},w_{-\theta})J(x_{-\theta}-y_{-\theta},v_{-\theta}-w_{-\theta})\\
    &\qquad\qquad\times \lla v-w,\nyx\rra\left(\frac{x_{-\theta}-y_{-\theta}}{|x-y|}\right)_iy\d n\d w\d \theta,\grad_x\phi(x,v)\bigg)_{x,v}.
\end{align*}

\end{proof}
If we additionally assume $J(x-y,v-w)$ is radially symmetric, i.e. $J(x-y,v-w)=J(|x-y|,|v-w|,\cos(\widehat{x-y,v-w}))$, which is true for Boltzmann collision kernel. $J(x_{-\theta}-y_{-\theta},v_{-\theta}-w_{-\theta})$ can be further simplified by $\jxy$.

\subsection{Energy current for Povzner collision integral: Proof of \eqref{Energy current for Povzner collision integral}}
\begin{proof}
Using the test function $\phi(x,v)=|v|^2\varphi(x,v)$ in the statement of Lemma \ref{lem:Povzner}, we obtain:
\begin{multline*}
    \int_{(\R^3)^4}f(x,v)f(y,w)\jxy  \\
    \times \big[|v'|^2\varphi(x,v')-|v|^2\varphi(x,v)\big]\d x\d v\d y\d w.
\end{multline*}
Then, with the help of \eqref{pre-to-post collision transformation}, we split the difference $|v'|^2\varphi(x,v')-|v|^2\varphi(x,v)$ into four parts,
\begin{multline*}
\underbrace{|v|^2[\varphi(x,v')-\varphi(x,v)]}_{=:\varphi_1}-\underbrace{\lla v,\nyx\rra^2[\phi(x,v')-\varphi(x,v)]}_{=:\varphi_2}\\
+\underbrace{\lla w,\nyx\rra^2[\varphi(x,v')-\varphi(x,v)]}_{=:\varphi_3}\\
-\underbrace{\left(\lla v,\nyx\rra ^2-\lla w,\nyx\rra^2\right)\varphi(x,v)}_{=:\varphi_4},
\end{multline*}
since
\begin{multline*}
    |v'|^2=\left|v-\lla v-w,\nyx\rra \nyx\right|^2\\=|v|^2-\lla v,\nyx\rra^2+\lla w,\nyx\rra^2.
\end{multline*}
Following this, we use the result \eqref{Mass current for Povzner collision integral}, which we obtained in the last section. It shows that
\begin{multline*}
 \int_{(\R^3)^4}f(x,v)f(y,w)\jxy  (\varphi_1-\varphi_2+\varphi_3)\d x\d v\d y\d w\\
 =\left(-\int_{\R^3}\Jj_0^P\bigg[(|v|^2-\lla v,\nyx\rra^2)f,f\right]\\+\Jj_0^E\left[f,\lla w,\nyx\rra^2f\right]\d y,\grad_v\phi\bigg)_{x,v}.
\end{multline*}

By the same process showed before, We perform the change of variables $(v,w,x,y)\mapsto (w,v,y,x)$ in $\varphi_4$ term, then symmetrize it,

\begin{multline*}
 \int_{(\R^3)^4}f(x,v)f(y,w)\jxy \lla v-w,\nyx\rra \varphi_4\d x\d v\d y\d w\\
 = \frac{1}{2}\int_{(\R^3)^4}f(x,v)f(y,w)\jxy \lla v-w,\nyx\rra \lla v+w,\nyx\rra\\
 \times \big[\varphi(x,v)-\varphi( y,w)\big]\d x\d v\d y\d w.
\end{multline*}
Therefore, we replace $\varphi(x,v)-\varphi(y,w)$ by \eqref{pf rotation in x and v Povzner},
\begin{multline*}
    \int_{(\R^3)^4}f(x,v)f(y,w)\jxy \varphi_4\d x\d v\d y\d w\\
    = -\int_{(\R^3)^4}\int^{\pi/2}_0\jxy \lla v-w,\nyx\rra \lla v+w,\nyx\rra\\
    \times f(x,v)f(y,w)\la y_\theta,\grad_x\phi(x_{\theta},v_\theta)\ra\d x\d v\d y\d w\d \theta\\
  - \int_{(\R^3)^4}\int^{\pi/2}_0\jxy \lla v-w,\nyx\rra \lla v+w,\nyx\rra\\
    \times f(x,v)f(y,w)\la w_\theta,\grad_v\phi(x_\theta,v_\theta)\ra\d x\d v\d y\d w\d \theta.
\end{multline*}
Finally, we rotate both the positions and velocities by $-\theta$, 
\begin{multline*}
    \int_{(\R^3)^4}f(x,v)f(y,w)\jxy \varphi_4\d x\d v\d y\d w\\
    = \bigg(-\int_{(\R^3)^2}\int^{\pi/2}_0f(x_{-\theta},v_{-\theta})f(y_{-\theta},w_{-\theta})J(x_{-\theta}-y_{-\theta},v_{-\theta}-w_{-\theta})\\
    \times \lla v-w,\nyx\rra \lla v+w,\nyx\rra y\d y\d w\d \theta,\grad_x\phi(x,v)\bigg)_{x,v}\\
  +\bigg(- \int_{(\R^3)^2}\int^{\pi/2}_0f(x_{-\theta},v_{-\theta})f(y_{-\theta},w_{-\theta})J(x_{-\theta}-y_{-\theta},v_{-\theta}-w_{-\theta})\\
    \times \lla v-w,\nyx\rra \lla v+w,\nyx\rra w\d y\d w\d \theta,\grad_v\phi(x,v)\bigg)_{x,v}.
\end{multline*}

In summary, we have
\begin{align*}
    &\left(|v|^2P[f,f](x,v),\varphi(x,v)\right)_{x,v}=(-\J_0^P[|v|^2f,f],\grad_v\phi)_{x,v}\\
    &-\left(-\int_{\R^3}\Jj_0^P[\lla v,\nyx\rra^2f,f]\d y,\grad_v\varphi\right)_{x,v}\\
    &+\left(-\int_{\R^3}\Jj_0^P[f,\lla w,\nyx\rra^2f]\d y,\grad_v\varphi\right)_{x,v}\\
    &-\frac{1}{2} \bigg(-\int_{(\R^3)^2}\int^{\pi/2}_0f(x_{-\theta},v_{-\theta})f(y_{-\theta},w_{-\theta})J(x_{-\theta}-y_{-\theta},v_{-\theta}-w_{-\theta})\\
    &\qquad\times \lla v-w,\nyx\rra \lla v+w,\nyx\rra y\d y\d w\d \theta,\grad_x\phi(x,v)\bigg)_{x,v}\\
    &-\frac{1}{2} \bigg(- \int_{(\R^3)^2}\int^{\pi/2}_0f(x_{-\theta},v_{-\theta})f(y_{-\theta},w_{-\theta})J(x_{-\theta}-y_{-\theta},v_{-\theta}-w_{-\theta})\\
    &\qquad\times \lla v-w,\nyx\rra \lla v+w,\nyx\rra w\d y\d w\d \theta,\grad_v\phi(x,v)\bigg)_{x,v},
\end{align*}
which is the weak formulation of \eqref{Energy current for Enskog collision integral}.
\end{proof}

\section{Conclusion}

This paper serves as a complement to \cite{CCG2024}. While the general framework developed there does not cover all delocalized collision models, we have here established conservative formulations for two cases that fall outside the framework given in \cite{CCG2024}, namely the orginal Povzner and standard Enskog collision integrals. A key feature of our result is that it does not rely on asymptotic expansions around equilibrium: the conservative form remains valid for distributions far from equilibrium, thereby extending its mathematical and physical relevance. This provides new insight into the structure of kinetic models where standard perturbative approaches fail, and reinforces the role of conservation and entropy principles in non-equilibrium regimes.

As a perspective, and in analogy with \cite{CCG2024}, we plan to extend these results to multispecies systems. In particular, we aim to investigate how the delocalized collision framework can be adapted to the study of dense gases.

\section*{Acknowledgements}
The author acknowledges financial support from the École Doctorale de Sciences Mathématiques de Paris Centre.  
This research was carried out at Sorbonne University.  
The author is grateful to Professor François Golse for proofreading the manuscript.

\bibliographystyle{alpha}
\bibliography{ConservativeForm}

@article{Villani1999,
author = {Villani, Cédric},
journal = {ESAIM: Mathematical Modelling and Numerical Analysis - Modélisation Mathématique et Analyse Numérique},
language = {eng},
number = {1},
pages = {209-227},
publisher = {Dunod},
title = {Conservative forms of Boltzmann's collision operator : Landau revisited},
url = {http://eudml.org/doc/193912},
volume = {33},
year = {1999},
}

@ARTICLE{Resibois1978,
  title    = "H-theorem for the (modified) nonlinear Enskog equation",
  author   = "Resibois, P",
  journal  = "Journal of Statistical Physics",
  volume   =  19,
  number   =  6,
  pages    = "593--609",
  month    =  dec,
  year     =  1978
}

@book{Landau10,
  title={Physical Kinetics: Volume 10},
  author={Pitaevskii, L.P. and Lifshitz, E.M.},
  number={ 10 },
  isbn={9780080570495},
  lccn={80042162},
  url={https://books.google.fr/books?id=DTHxPDfV0fQC},
  year={2012},
  publisher={Elsevier Science}
}

@article{Enskog1922,
  title={Kinetische Theorie der Waermw Leitung, Reubung und Selbstdiffussion in Gewissen Verdichteten Gasen und Fluessigkeiten},
  author={Enskog, David},
  journal={Annalen der Physik},
  volume={68},
  pages={1--86},
  year={1922},
  publisher={Wiley Online Library}
}

@article{VanBeijeren1983,
  title = {Equilibrium Distribution of Hard-Sphere Systems and Revised Enskog Theory},
  author = {van Beijeren, Henk},
  journal = {Phys. Rev. Lett.},
  volume = {51},
  issue = {17},
  pages = {1503--1505},
  numpages = {0},
  year = {1983},
  month = {Oct},
  publisher = {American Physical Society},
  doi = {10.1103/PhysRevLett.51.1503},
  url = {https://link.aps.org/doi/10.1103/PhysRevLett.51.1503}
}

@article{VANBEIJEREN1973225,
title = {The modified Enskog equation for mixtures},
journal = {Physica},
volume = {70},
number = {2},
pages = {225-242},
year = {1973},
issn = {0031-8914},
doi = {https://doi.org/10.1016/0031-8914(73)90247-4},
url = {https://www.sciencedirect.com/science/article/pii/0031891473902474},
author = {H. {Van Beijeren} and M.H. Ernst}
}

@article{VANBEIJEREN1973437,
title = {The modified Enskog equation},
journal = {Physica},
volume = {68},
number = {3},
pages = {437-456},
year = {1973},
issn = {0031-8914},
doi = {https://doi.org/10.1016/0031-8914(73)90372-8},
url = {https://www.sciencedirect.com/science/article/pii/0031891473903728},
author = {H. {Van Beijeren} and M.H. Ernst}
}

@article{povzner1965,
  title        = {The Boltzmann Equation in the Kinetic Theory of Gases},
  author       = {Povzner, A. J.},
  journal      = {American Mathematical Society Translations},
  year         = {1965},
  volume       = {47},
  series       = {Series 2},
  pages        = {193--214},
}

@article{FORNASIER2011,
title = {Fluid dynamic description of flocking via the Povzner–Boltzmann equation},
journal = {Physica D: Nonlinear Phenomena},
volume = {240},
number = {1},
pages = {21-31},
year = {2011},
issn = {0167-2789},
doi = {https://doi.org/10.1016/j.physd.2010.08.003},
url = {https://www.sciencedirect.com/science/article/pii/S0167278910002344},
author = {Massimo Fornasier and Jan Haskovec and Giuseppe Toscani}
}

@misc{CCG2024,
  title        = {Local Conservation Laws and Entropy Inequality for Kinetic Models with Delocalized Collision Integrals}, 
  author       = {Frédérique Charles and Zhe Chen and François Golse},
  year         = {2024},
  eprint       = {2412.16646},
  archivePrefix= {arXiv},
  primaryClass = {math-ph},
  url          = {https://arxiv.org/abs/2412.16646}, 
  note         = {To appear in Math. Models and Methods in the Appl. Sci.}
}

@book{Cercignani1994,
  author       = {Carlo Cercignani},
  title        = {The Mathematical Theory of Dilute Gases},
  year         = {1994},
  publisher    = {Springer},
  series       = {Applied Mathematical Sciences},
  volume       = {106},
  isbn         = {978-0-387-94294-7},
  doi          = {10.1007/978-1-4612-4286-3},
  url          = {https://link.springer.com/book/10.1007/978-1-4612-4286-3}
}

@book{chapman1990,
  title={The Mathematical Theory of Non-uniform Gases: An Account of the Kinetic Theory of Viscosity, Thermal Conduction and Diffusion in Gases},
  author={Chapman, S. and Cowling, T.G.},
  isbn={9780521408448},
  lccn={70077285},
  series={Cambridge Mathematical Library},
  url={https://books.google.fr/books?id=Cbp5JP2OTrwC},
  year={1990},
  publisher={Cambridge University Press}
}

@article{Goldstein1997,
    author = {Goldstein, Patricia and Garc\'ia-Col\'in, L. S.},
    title = {An H-theorem for the Enskog equation of a binary mixture of dissimilar hard spheres},
    journal = {The Journal of Chemical Physics},
    volume = {106},
    number = {1},
    pages = {236-246},
    year = {1997},
    month = {01},
    issn = {0021-9606},
    doi = {10.1063/1.473027},
    url = {https://doi.org/10.1063/1.473027},
    eprint = {https://pubs.aip.org/aip/jcp/article-pdf/106/1/236/19207714/236\_1\_online.pdf},
}
\end{document}